\documentclass[11pt,a4paper,twoside]{amsart}

\usepackage{amsmath,amssymb,amsthm}
\usepackage{mathtools}
\usepackage{mathrsfs}

\usepackage{inputenc}
\usepackage{tikz} 
\usepackage{float}
\floatstyle{boxed} 
\usepackage{graphicx,subfigure}
\usepackage{caption}
\usepackage[T1]{fontenc} 
\usepackage{enumitem}
\usepackage{xcolor}
\usepackage[colorlinks=true, linkcolor=magenta, citecolor=cyan, urlcolor=blue,hyperfootnotes=false]{hyperref}
\usepackage{cleveref}

\usepackage[style=numeric,firstinits=true,sorting=nyt,backend=bibtex,maxnames=99]{biblatex}
\usepackage{etoolbox}
\usepackage{upgreek}

\def\@begintheorem#1#2{\par\bgroup{\sc #1 \ #2. }  \it \\\ignorespace }
\def\@opargbegintheorem#1#2#3{\par\bgroup{\sc #1\ #2 \ (#3).}  \it  \ignorespace}
\def\@endtheorem{\egroup}

\theoremstyle{plain}
\newtheorem{theorem}{Theorem}[section]

\theoremstyle{definition}

\theoremstyle{remark}
\newtheorem{remark}[theorem]{Remark}

\theoremstyle{plain}

\theoremstyle{plain}
\newtheorem{lemma}[theorem]{Lemma}

\theoremstyle{plain}

\theoremstyle{plain}

\numberwithin{equation}{section}

\newcommand{\zb}{{\bar{z}}}
\newcommand{\C}{{\mathbb C}}

\newcommand{\N}{{\mathbb N}}
\newcommand{\D}{\mathcal{D}}
\newcommand{\Kel}{\mathcal{K}}

\newcommand{\R}{{\mathbb R}}

\newcommand{\V}{\mathbb{V}}

\newcommand{\supp}{{\rm supp}}

{\left\lbrace\begin{array}{@{}l@{}}}%
{\end{array}\right.}

\DeclarePairedDelimiter{\abs}{\lvert}{\rvert}
\DeclarePairedDelimiter{\norm}{\lVert}{\rVert}

\newcommand{\E}{\textbf{E}}

\newcommand{\ignora}[1]{}

\title{On the  strong unique continuation property for the Dirac operator}
\date{\today}
\author[B. Cassano]{Biagio Cassano}
\address{B.~Cassano, Department of Mathematics and Physics, %
  Universit\`{a} degli  Studi della Campania, %
  viale Lincoln 5, 81100, Caserta, Italy}
\email{biagio.cassano@unicampania.it}

\usepackage{csquotes}

\bibliography{biblio.bib}
\begin{document}
\begin{abstract}
In \cite{de1999strong,kalf1999note}, the strong unique continuation property from the origin is established for $H_{loc}^1$-solutions to the massless Dirac differential inequality $|\D_n u | \leq \frac{C}{\abs{x}}|u|$, in dimension $n\geq 2$ and with $C<\frac12$. We show that $\frac12$ is the largest possibile constant in this result, providing an example in $\R^2$ of a (non-trivial) solution of the inequality. Also, we show properties of unique continuation from the origin for solutions to the inequality $|\D_n u | \leq \frac{C}{ \abs{x}^{\gamma}}|u|$, for $\gamma>1$, $C>0$.
Finally, we establish the strong unique continuation property for the Dirac operator from the point at infinity.
\end{abstract}
\maketitle



\section{Introduction}
In this paper, we discuss properties of unique continuation for the Dirac operator.
In general, a partial differential equation or inequality defined in a connected domain has the \emph{unique continuation property} if all its solutions  that vanish in a nonempty open set vanish everywhere in the domain. Establishing the unique continuation property is instrumental to prove the uniqueness of the solution, or can be used to show the absence of embedded eigenvalues for the differential operator.
We say that a partial differential equation or inequality has the \emph{strong unique continuation property} if all its solutions that vanish at infinite order at a point $x_0$, that is
\begin{equation}\label{eq:vanishing.infinite.order}
\lim_{R\to 0} \frac{1}{R^k} \int_{\{|x-x_0|<R\}} |u(x)|^2 \, dx = 0 \quad \text{ for all }k \in \N,
\end{equation}
vanish in a neighborhood of $x_0$.
There is a large literature on the unique continuation property for the Dirac operator: we refer to 
\cite{jerison1986carleman,vogelsang1987absence,mandache1994some,kim1995carleman,de1999strong,kalf1999note,booss2000unique, salo2009carleman, esteban2021dirac, jeong2022carleman} and references therein.
We refer also to \cite{jerison1985unique,stein1985appendix,garofalo1987unique,wolff1990unique} for results on unique continuation for the Laplace operator that are related to our topic, to \cite{seo2015unique} for the fractional laplacian, and to \cite{pan2024unique} where the operator
 $\bar\partial$ is considered.

In order to state our results, we remind the definition and some
well known properties of the Dirac operator.
For $n\geq 2$, let $N := 2^{\lfloor \frac{n+1}{2}\rfloor}$, where
$\lfloor \cdot \rfloor$ denotes the integer part of a real number.
It is well known that there exist Hermitian matrices
$\alpha_1,\dots,\alpha_n, \alpha_{n+1} \in \C^{N\times N}$
that satisfy the anticommutation relations
\begin{equation}
  \label{eq:anticommutation}
  \alpha_j \alpha_k +
  \alpha_k \alpha_j =
  2 \delta_{j,k} \mathbb{I}_N, 
  \quad 
  1\leq j,k \leq n+1,
\end{equation}
where $\delta_{j,k}$ is the Kronecker delta and $\mathbb{I}_N$ is the 
$N \times N$ identity matrix; these matrices form a representation of the Clifford algebra of the
Euclidean space (see e.g.~\cite{boussaid2019nonlinear}). 
The Dirac operator acts on functions $\psi : \R^n \to \C^{N}$ and
it is defined as follows:
\begin{equation}\label{eq:defn.Dnm}
  \D_{n,m} \psi := -i \alpha \cdot \nabla \psi + m \alpha_{n+1}\psi
  := -i \sum_{j=1}^n \alpha_j \partial_j \psi + m \alpha_{n+1}\psi,
\end{equation}
where $m \in \R$. From \eqref{eq:anticommutation} we have that
$\D_{n,m}^2 = (-\Delta + m^2)\mathbb{I}_{N}$.
We remark that  $\D_{n,m}$ is self-adjoint in $L^2(\R^n; \C^N)$
with domain $H^1(\R^n; \C^N)$. 
In the following, when $m=0$ we denote $\D_{n}:=\D_{n,0}$.

In the definition of the Dirac operator, an explicit choice for the matrices $\alpha_j$ in \eqref{eq:defn.Dnm}
 is not required, since all the results can be obtained
using their fundamental property \eqref{eq:anticommutation};
however it is convenient to give one standard choice in the
case $n=2$, for the purpose of making computations. 
We choose explicitely
$(\alpha_1,\alpha_2,\alpha_3):=(\sigma_1,\sigma_2,\sigma_3)$, where
$\sigma_j$ are the \emph{Pauli matrices}
\begin{equation}\label{eq:paulimatrices}
\quad{\sigma}_1 =
\begin{pmatrix}
0 & 1\\
1 & 0
\end{pmatrix},\quad {\sigma}_2=
\begin{pmatrix}
0 & -i\\
i & 0
\end{pmatrix},
\quad{\sigma}_3=
\begin{pmatrix}
1 & 0\\
0 & -1
\end{pmatrix}.
\end{equation}
Using the usual identification
$(x_1,x_2)\in\R^2 \mapsto z=x_1 + i x_2 \in \C$, the two-dimensional Dirac operator is then
\begin{equation}\label{eq:defn.D}
	\D_{2,m} := 
	- i \sigma_1 \partial_1 - i \sigma_2 \partial_2 + m \sigma_3 =
	\begin{pmatrix}
	m & -2i \partial_z \\
	-2i \partial_{\bar z} & -m 
         \end{pmatrix},
\end{equation}
where we denote $2\partial_z = \partial_1 -i \partial_2$,
$2\partial_{\zb} = \partial_1 + i \partial_2$.

De Carli and Okaji \cite{de1999strong} and Kalf and Yamada \cite{kalf1999note} show the following result of strong unique continuation for the Dirac Operator.
\begin{theorem}[\cite{de1999strong,kalf1999note}]\label{thm:DCOKY}
Let $n\geq 2$ and $\Omega \subset \R^n$ be a domain containing the origin. Let $ u \in H^1_{loc} (\Omega;\C^N)$ be a solution to the differential inequality
\begin{equation}\label{eq:main}
|\D_n u(x)| \leq \frac{C}{|x|}| u(x) |,  \quad x \in \Omega\setminus\{0\},
\end{equation}
for  
\begin{equation}\label{eq:sharpness}
C \in \left(0,\frac12 \right),
\end{equation}
 and assume that $u$ vanishes at infinite order at the origin, that is
\begin{equation}\label{eq:vanishing.infinite.order.origin}
\lim_{R\to 0} \frac{1}{R^k} \int_{\{|x|<R\}} |u(x)|^2 \, dx = 0, \quad \text{ for all }k \in \N.
\end{equation}
Then, $u$ vanishes in $\Omega$.
\end{theorem}
\begin{remark}\label{rmk:lightening}
The interested reader may also refer to \cite[Theorem 2.1]{kalf2015dirac} (see also \cite{kalf2019erratum}), where
it is shown that if $u \in H_{loc}^{1}(\Omega ;\C^N) \cap L^2(\Omega,|x|^{-1/2}dx;\C^N)$ is solution to \eqref{eq:main} for $C \in (0,\frac{n-1}{2})$, then $u$ vanishes in $\Omega$.
\end{remark}
\begin{remark}
Perturbations of the order $\abs{x}^{-1}$ are critical for the Dirac operator from the point of view of scaling. For more general results in the scaling critical setting we refer to \cite{esteban2021dirac}, where the strong unique continuation property from one point is proven for solutions to $|\D_{3,0} u| \leq V |u|$, assuming that $V$ is locally in the Lorentz space $L^{3,\infty}$ and has a norm smaller than a universal constant. The smallness on the norm of $V$ is required also in analogous results for the fractional laplacian \cite{seo2015unique} and for the laplacian \cite{stein1985appendix} and is analogous to the condition \eqref{eq:sharpness} in \Cref{thm:DCOKY}.
\end{remark}

The matter of the optimality of the constant $\frac12$ in \eqref{eq:sharpness} is investigated in \cite{de1999strong}: the Authors show that the best constant for the validity of \Cref{thm:DCOKY} can not be bigger than $1$, producing an example  in $\R^2$ of a (non-trivial) function $u$ vanishing at infinite order at the origin and verifying \eqref{eq:main} with $C>1$ arbitrarily close to $1$. To do so, they adapt an example by Alinhac and Baouendi \cite{alinhac1994counterexample} related to properties of strong unique continuation for the Laplacian. 
In the following theorem, main result of this paper, we construct a (non-trivial) solution to \eqref{eq:main} in $\R^2$, showing so that in this case the best constant in  \eqref{eq:sharpness} of Theorem \ref{thm:DCOKY} is indeed $\frac12$.
\begin{theorem}\label{thm:main}
For all $C>\frac12$ there exists a (non-trivial)  function $u \in H^1(\R^2;\C^2) \cap C^\infty(\R^2 \setminus \{0\};\C^2)$ vanishing at infinite order at the origin as in \eqref{eq:vanishing.infinite.order.origin} such that \eqref{eq:main} holds true in $\R^2\setminus\{0\}$.
\end{theorem}

We complement \Cref{thm:DCOKY} with the description of properties of unique continuation for differential inequalities with a more singular behaviour of $\D_n u$ in the origin. 
\begin{theorem}\label{theo:gamma}
Let $n\geq 2$, $\rho>0$ and $\Omega_\rho := \{x\in\R^n :  |x| < \rho\}$. Let $u \in H_{loc}^1(\Omega_\rho;\C^N)$ such that 
\begin{equation}\label{eq:gamma}
\abs{\D_n u (x) } \leq \frac{C}{\abs{x}^{\gamma}} \abs{u(x)} \quad \text{ in }\Omega_\rho\setminus\{ 0 \},
\end{equation}
for $C>0$ and $\gamma>1$. Then, there exists $\tau_0>0$ such that if
\begin{equation}\label{eq:theo.power.gamma}
\exp\left[\frac{\tau}{\abs{x}^{2\gamma-2}}\right]u \in L^2 (\Omega_\rho;\C^N) \quad \text{ for all }\tau>\tau_0 
\end{equation}
then $u$ vanishes in a neighborhood of $0$.
\end{theorem}
The exponent $2\gamma -2$ in \eqref{eq:theo.power.gamma} is sharp in $\R^n$ for $n=2,3$. Indeed, in the following theorem we provide examples of (non-trivial) solutions to \eqref{eq:gamma} that have the prescripted decay in the origin.
\begin{theorem}\label{theo:gamma.example}
For all $\gamma>1$ and $n\in\{2,3\}$, there exist a (non-trivial) function $u \in H^1(\R^n;\C^N) \cap C^\infty(\R^n\setminus\{0\};\C^N)$ such that for all $x \in \R^n\setminus\{0\}$
\begin{align}
\label{eq:gamma.sharp.u}
\abs{u(x)} & \leq C_1 \exp\left[-\frac{C_2}{|x|^{2\gamma-2}}\right], \\
\label{eq:gamma.sharp.V}
\abs*{\D_n u(x)} & \leq
\begin{cases}
C_3 |x|^{-\gamma} \abs{u(x)} \quad &\text{ for }n=2, \\
C_3 |x|^{-\gamma} \abs*{\log\abs{x}}^3 \abs{u(x)} \quad &\text{ for }n=3,
\end{cases}
\end{align}
for some $C_1,C_2,C_3>0$.
\end{theorem}

In the proof of \Cref{theo:gamma} and \Cref{theo:gamma.example}, we use a Kelvin transformation adapted to the Dirac operator (see the following \eqref{eq:defn.Kelvin.transformation}) to exploit analogous results of unique continuation from the point at infinity in \cite{cassano2022sharp}, where a version of the Landis' problem was investigated for the Dirac operator. In the last result of this paper, using the strategy of the proof of \Cref{theo:gamma} and \Cref{theo:gamma.example}, we use \Cref{thm:DCOKY} and \Cref{thm:main} to show the following sharp result of strong unique continuation from the point at infinity, so complementing the results in \cite{cassano2022sharp}.

\begin{theorem}\label{thm:SUCinf}
Let $n\geq 2$, $\rho>0$ and $\widetilde\Omega_\rho:= \{ x \in \R^n: |x| > \rho\}$. Let $ u \in H^1_{loc} (\widetilde\Omega_\rho ; \C^N)$ be a solution to the differential inequality
\begin{equation}\label{eq:main.inf}
|\D_n u(x)| \leq \frac{C}{|x|}| u(x) |  \quad \text{ in } \widetilde\Omega_\rho,
\end{equation}
for  $C \in \left(0,\frac12 \right)$,  and assume that $u$ vanishes at infinite order at infinity, that is
\begin{equation}\label{eq:vanishing.infinite.order.inf}
\lim_{R\to +\infty} R^k \int_{\{|x|>R\}} |u(x)|^2 \, dx = 0, \quad \text{ for all }k \in \N.
\end{equation}
Then, $u$ has compact support. 
Moreover, for any $C>\frac12$ there exists  $C^\infty(\R^2;\C^2)$ not compactly supported and vanishing at infinite order at the infinity as in \eqref{eq:vanishing.infinite.order.inf} such that \eqref{eq:main.inf} holds true.
\end{theorem}

\begin{remark}
  In \Cref{thm:main}, \Cref{theo:gamma.example} and \Cref{thm:SUCinf}, the sharpness of the results is proven exhibiting potentials $\V \in C^\infty(\R^n\setminus\{0\};\C^{N\times N})$ such that $\D_n u = \V u$ and the required decay assumptions on  $u$ and $\V$ hold. The constructed potentials $\V$ are not symmetric: 
  it would be interesting to understand what could be the analogous to \Cref{thm:DCOKY}, \Cref{theo:gamma} and \Cref{thm:SUCinf} allowing only the presence of symmetric potentials. For an extension of \Cref{thm:DCOKY} in this direction, see \cite[Theorem 2.1]{de1999strong}.
\end{remark}
\subsection*{Structure of the paper}
\Cref{thm:main} is proven in \Cref{sec:controesempio}. \Cref{theo:gamma}, \Cref{theo:gamma.example} and  \Cref{thm:SUCinf} are proven in \Cref{sec:kelvin}.

\subsection*{Acknowledgments}
B.C.~acknowledges Nabile Boussa\"id and Piero D'Ancona for very fruitful discussions.
B.C.~is member of GNAMPA (INDAM) and has been partially supported by the PRIN 2022 project “Anomalies in partial differential equations and applications” CUP D53C24003370006 .

\section{Proof of \Cref{thm:main}}
\label{sec:controesempio}
In this section we make use of polar coordinates, writing $ z \in \C$ as $z  = r e^{i \theta}$, being $r = |z|>0$ and $\theta = \text{Arg}(z) \in (-\pi,\pi]$ its principal argument.
We recall that 
\begin{equation}\label{eq:partialz.in.radial}
\frac{\partial}{\partial z}  = \frac{e^{-i\theta}}{2} \left( \frac{\partial}{\partial r} - \frac{i}{r} \frac{\partial}{\partial \theta} \right), 
\quad
\frac{\partial}{\partial \bar{z}}  = \frac{e^{i\theta}}{2} \left( \frac{\partial}{\partial r} + \frac{i}{r} \frac{\partial}{\partial \theta} \right).
\end{equation}

In this section we simplify the notations and write $\D:=\D_{2,0}$.

Let $\epsilon>0$. We construct (non-trivial) functions $u \in H^1(\R^n;\C^N) \cap C_c^\infty(\R^2\setminus\{0\};\C^2)$ and $\V \in C^\infty(\R^2\setminus \{0\};\C^{2\times 2})$ such that for all $z \in \R^2\setminus\{0\}$
\begin{gather}
\notag
 \D u(z) = \V(z) u (z),\\
 \label{eq:decayV}
 |\V (z) |  = \sup_{w\in\C^2}\frac{\abs{\V(z) w}}{|w|} \leq \left(\frac12+\epsilon\right)\frac{1}{|z|}.
\end{gather}
The strategy of the proof is the
following:
\begin{itemize}
\item in \Cref{sec:preliminary}: we define the useful function $\chi_\delta$; we define the appropriate sequence $(\rho_k)_{k \in \N}$, such that  $\rho_k \searrow
  0$; 
  we define the functions $\mathbf{E}_k$;
\item in \Cref{sec:defn.annulus.2D}: for $k \geq k_0$, we define the functions $u_k$ and
  $\V_k$ in the annulus $\{\rho_{k+1} \leq \abs{z}\leq
  \rho_{k}\}$;
    \item in \Cref{sec:defn.infinity.2D}: we define the functions $u_0$ and $\V_0$ in   $\{\abs{z}\geq
  \rho_{k_0}\}$;  
\item  in \Cref{sec:defn.u.V.2D}: we define $u$ and $\V$ glueing together the
  functions $u_0, \V_0$ and  $u_k, \V_k$, for $k\geq k_0$; we check the behaviour in the origin of the function $u$.
\end{itemize}

\subsection{Preliminary definitions}\label{sec:preliminary}
In the following lemma we show the existence of a function with appropriate properties.
\begin{lemma}\label{lem:cut.off} 
Let $\delta \in (0,\frac14)$.
There exists $\chi_\delta \in C^\infty(\R)$ such that 
\begin{enumerate}[label=(\roman*)]
    \item\label{itm:1} $\chi_\delta(s) = 0$ for $s\leq 0$, $\chi_\delta(s) \in [0,1]$ for $s \in [0,1]$, $\chi_\delta(s) = 1$ for $s\geq 1$;
\item\label{itm:2} $\supp \chi_\delta' \subset [0,1]$, $\chi_\delta'(s) \geq 0$ for $s \in [0,1]$, $\norm{\chi_\delta'}_{L^\infty(\R)} = \chi_\delta'(\frac12) = 1+\delta$;
\item\label{itm:3} $\chi_\delta(s) \leq s$ for $s \in [0,\frac12]$, $\chi_\delta(\frac12) = \frac12$, $\chi_\delta(s) \leq (1+\delta) s -\frac{\delta}2$ for $s \in [\frac12,1]$.
\end{enumerate}
\end{lemma}
\begin{proof}
We exhibit a function that has all the required properties.
Let $0< \delta< \frac14$ and let $\chi:\R \to \R$ as follows
\begin{equation*}
\chi(s) :=
\begin{cases}
0 \quad & \text{ for }s\leq \delta,
\\
\frac{s-\delta}{1-2\delta} \quad & \text{ for } s \in [\delta, 1-\delta],
\\
1 \quad & \text{ for }s \geq 1-\delta.
\end{cases}
\end{equation*}
Let $\psi \in C^\infty(\R)$ be an even function such that  $\supp \psi \subset [-1,1]$  and $\int_{-1}^{1} \psi(s) ds = 1$. We let $\psi_\delta := \frac{1}{\delta}\psi(\frac{\cdot}{\delta})$ and we set $\chi_\delta := \chi \ast \psi_\delta \in C^\infty(\R)$. As shown by direct computations, \ref{itm:1} -- \ref{itm:3} hold true.
\end{proof}
We let $\delta \in (0,\frac14)$ such that 
\begin{equation}\label{eq:delta.small}
\delta^2 + \delta \leq \epsilon.
\end{equation}
From \Cref{lem:cut.off}, there exist a function $\chi_\delta \in C^\infty(\R)$ such that \ref{itm:1} -- \ref{itm:3} hold true.

We define the sequence $(\rho_k)_{k \in \N}$ as follows
\begin{equation}\label{eq:defn.rhok}
\rho_k := \frac{1}{e^{e^{k^2}}} \quad \text{ for all }k \in \N
\end{equation}
and, for $j=0,\dots,6$ we let
\begin{equation}\label{eq:defn.rhokj}
\rho_{k,j}:=\frac{1}{e^{e^{\left(k+\frac{j}{6}\right)^2}}}. 
\end{equation}
We let $k_0 \in \N$ be a big number, to be chosen later in the proof in such a way that \eqref{eq:choice.k0.21}, \eqref{eq:choice.k0.32}, \eqref{eq:choice.k0.43}, \eqref{eq:choice.k0.54}, \eqref{eq:choice.k0.0} hold true.
We have that
\begin{equation}\label{eq:decomposition.R2}
\begin{split}
\R^2  & = \{z \in \R^2 :  |z| \geq \rho_{k_0}    \}  \cup \bigcup_{k \geq k_0} \{z \in \R^2 : \rho_{k+1} \leq |z| \leq \rho_{k} \}
\\ 
& =  \{z \in \R^2 : |z| \geq \rho_{k_0}   \}  \cup \bigcup_{k \geq k_0} \bigcup_{ j= 0}^5 \{z \in \R^2 : \rho_{k,j+1} \leq |z| \leq \rho_{k,j} \}.
\end{split}
\end{equation}

For $k\in\N$, we set
\begin{equation}\label{eq:defn.Ek}
  \E_k(z) :=
  \begin{pmatrix}
    0 \\ \bar{z}^k
  \end{pmatrix}
  \text{ if $k$ is even}, 
  \quad
  \E_k(z) :=
  \begin{pmatrix}
    z^k \\  0
  \end{pmatrix}
  \text{ if $k$ is odd}. 
\end{equation}
From \eqref{eq:defn.D}, we have that for all $ z \in \R^2$ and $k \in \N$:
\begin{equation}\label{eq:zeroEk}
  \D \, \mathbf{E}_k (z) = 0.
\end{equation}

\subsection{Definition of $u_k$ and $\V_k$ in the annulus $\{\rho_{k+1} \leq
  \abs{z}\leq\rho_{k}\}$}
\label{sec:defn.annulus.2D}
We let $k \geq k_0$.
In the proof in this subsection we assume $k$ to be even: the
case of odd $k$ can be treated analogously.

For $k\in\N$, $k \geq k_0$, in this section we construct functions 
\begin{equation}\label{eq:regolarityannulus}
  \begin{split}
    &u_k \in C^{\infty}(\{z \in \R^2 \colon \rho_{k+1} \leq \abs{z} \leq
    \rho_{k} \}; \C^2), \\
    &\V_k \in C^\infty(\{z \in \R^2 \colon \rho_{k+1} \leq \abs{z} \leq
    \rho_{k} \}; \C^{2\times 2}),
  \end{split}
\end{equation}
such that for all $\rho_k  \leq \abs{z} \leq \rho_{k+1}$
\begin{equation}
  \label{eq:main2dlocal}
  \mathcal D u_k (z) = \V_k (z) u_k(z) 
\end{equation}
and 
\begin{align}
\label{eq:decayVk}
& |\V_k (z)| \leq  \left(\frac{1}{2} +\epsilon \right) \frac{1}{|z|}, 
\\
  \label{eq:decayuk}
&  \abs{u_k(z)} \leq  2 |z|^k.
\end{align}

\subsubsection{Definition of the functions $u_k$ and $\mathbb V_k$ in $\{\rho_{k,1} \leq
  \abs{z}\leq\rho_{k,0}\}$ and $\{\rho_{k,6} \leq
  \abs{z}\leq\rho_{k,5}\}$}
    \label{subsec:10.2D}
  For $|z| \in [\rho_{k,6}, \rho_{k,5}] \cup  [\rho_{k,1}, \rho_{k,0}]$, we let
  \begin{equation*} u_k(z): = 
  \begin{cases}
    \mathbf{E}_{k+1}(z) \quad &\text{for $r \in [\rho_{k,6},\rho_{k,5}]$}, \\
    \mathbf{E}_{k}(z) \quad &\text{for $r \in [\rho_{k,1},\rho_{k,0}]$}; 
  \end{cases}
  \quad \V_k(z):= 
  \begin{pmatrix}
  0 & 0 \\
  0 & 0
  \end{pmatrix}.
\end{equation*}
Thanks to \eqref{eq:zeroEk}, \eqref{eq:regolarityannulus} -- \eqref{eq:decayuk} hold true.

\subsubsection{Definition of the functions $u_k$ and $\mathbb V_k$ in $\{\rho_{k,2} \leq
  \abs{z}\leq\rho_{k,1}\}$}
    \label{subsec:21.2D}
We let 
\begin{align}
& \varphi_k(r):=\chi_\delta \left( c_k \left(1 - \frac{\log \rho_{k,1}}{\log r}\right)\right) \quad \text{ for all } r>0,
\\
\label{eq:defn.ck}
& c_k:= \left(1 - \frac{\log \rho_{k,1}}{\log \rho_{k,2}}\right)^{-1}>0.
\end{align}
We have that $\varphi_k \in C^\infty(0,+\infty)$, $\varphi(\rho_{k,2})=1$, $\varphi(\rho_{k,1})=0$, $\varphi_k$ is decreasing and $\supp \varphi_k' \subset [\rho_{k,2}, \rho_{k,1}]$.

For $z  = r e^{i\theta}$, $r \in  [\rho_{k,2},\rho_{k,1}]$ and $\theta \in [-\pi,\pi)$, we set
\begin{align}
\notag
 u_k (z)  & := |z|^{\frac12 \varphi_k(r)}     \mathbf{E}_{k}(z), \\
\label{eq:defn.Vk.21}
\V_k(z) & := 
\begin{pmatrix}
0 & \frac{-i e^{-i\theta}}{2} \frac{d}{dr}(\varphi_k(r) \log r) \\
0 & 0
\end{pmatrix}.
\end{align}
To show \eqref{eq:main2dlocal}, we compute explicitely $\D u_k$. From \eqref{eq:defn.D} and \eqref{eq:defn.Ek}, for all $|z| = r \in  [\rho_{k,2},\rho_{k,1}]$ we have
\begin{equation*}
\D u_k(z) = 	
\begin{pmatrix}
0 & -2i \partial_z \\
-2i \partial_{\bar z} & 0 
\end{pmatrix}
\begin{pmatrix}
0 \\  |z|^{\frac12 \varphi_k(r)} \bar{z}^k
\end{pmatrix}
=
\begin{pmatrix}
  -2i \partial_z \left(|z|^{\frac12 \varphi_k(r)}\right) \bar{z}^k \\ 0
\end{pmatrix}.
\end{equation*}
Thanks to \eqref{eq:partialz.in.radial}, 
\begin{equation*}
-2i \partial_z \left(|z|^{\frac12 \varphi_k(r)}\right) = -i e^{-i\theta} \frac{d}{dr}\left( r^{\frac12 \varphi_k (r)} \right)
=
\frac{-i e^{-i\theta}}{2} r^{\frac12 \varphi_k (r)} \frac{d}{dr}\left(\varphi_k(r) \log r\right),
\end{equation*}
so \eqref{eq:main2dlocal} holds true. 

To show \eqref{eq:decayVk}, we note that
\begin{equation}\label{eq:estimateV.1}
\begin{split}
& \left\vert \frac{-ie^{-i\theta}}{2} \frac{d}{dr} \left(\varphi_k(r) \log r \right)\right\vert
= \frac12 \left\vert \varphi_k'(r) \log r + \frac{\varphi_k(r)}{r} \right\vert
\\
& = \frac1{2r} \left\vert \chi_\delta' \left( c_k \left(1 - \frac{\log \rho_{k,1}}{\log r}\right)\right)\frac{c_k \log \rho_{k,1}}{\log r} + \chi_\delta \left( c_k \left(1 - \frac{\log \rho_{k,1}}{\log r}\right)\right)   \right\vert
\\
& = \frac1{2r} \left( \chi_\delta' \left( c_k \left(1 - \frac{\log \rho_{k,1}}{\log r}\right)\right)\frac{c_k \log \rho_{k,1}}{\log r} + \chi_\delta \left( c_k \left(1 - \frac{\log \rho_{k,1}}{\log r}\right)\right)   \right),
\end{split}
\end{equation}
where in the last equality we have used the fact that $\log r \leq \log \rho_{k,1}\leq \log \rho_0 <0$, $\chi_\delta \geq 0$ and $\chi_\delta' \geq 0$ thanks to \Cref{lem:cut.off}\ref{itm:1},\ref{itm:2}.
From \Cref{lem:cut.off}\ref{itm:2},\ref{itm:3}, if $c_k - \frac{c_k \log \rho_{k,1}}{\log r} \in [0,\frac12]$, we have
\begin{equation}\label{eq:estimateV.2}
\begin{split}
& \chi_\delta' \left( c_k \left(1 - \frac{\log \rho_{k,1}}{\log r}\right)\right) \frac{c_k \log \rho_{k,1}}{\log r} + \chi_\delta \left( c_k \left(1 - \frac{\log \rho_{k,1}}{\log r}\right)\right) 
\\
& \leq 
(1+\delta) \frac{c_k \log \rho_{k,1}}{\log r} + c_k - \frac{c_k \log \rho_{k,1}}{\log r} \leq c_k \left( \delta \frac{\log \rho_{k,1}}{\log r}  + 1 \right)
\\
& \leq (1+\delta) c_k.
\end{split}
\end{equation}
From \Cref{lem:cut.off}\ref{itm:2},\ref{itm:3}, if $c_k - \frac{c_k \log \rho_{k,1}}{\log r} \in [\frac12,1]$ we have
\begin{equation}\label{eq:estimateV.3}
\begin{split}
& \chi_\delta' \left( c_k \left(1 - \frac{\log \rho_{k,1}}{\log r}\right)\right) \frac{c_k \log \rho_{k,1}}{\log r} + \chi_\delta \left( c_k \left(1 - \frac{\log \rho_{k,1}}{\log r}\right)\right) 
\\
& \leq (1+\delta)\frac{c_k \log \rho_{k,1}}{\log r} + (1+\delta)\left( c_k - \frac{c_k \log \rho_{k,1}}{\log r} \right) - \frac{\delta}{2}
\\
& \leq (1+\delta) c_k.
\end{split}
\end{equation}
From \eqref{eq:defn.rhokj} and \eqref{eq:defn.ck}, $c_k = 1 + (e^{\frac{k}{3} + \frac{1}{12}}-1)^{-1}$, 
so the sequence $(c_k)_k$ is decreasing and  
\begin{equation}\label{eq:choice.k0.21}
1\leq c_k \leq 1+\delta \quad \text{ for }k\geq k_0,
\end{equation}
for $k_0 \in \N$ big enough.
Thanks to \eqref{eq:delta.small}, \eqref{eq:defn.Vk.21}, \eqref{eq:estimateV.1}  
 -- \eqref{eq:choice.k0.21} we conclude that for all $r = \abs{z} \in [\rho_{k,2}, \rho_{k,1}]$
\begin{equation}
\abs{\V_k(z)} \leq \abs*{\frac{-i e^{-i\theta}}{2} \frac{d}{dr}(\varphi_k(r) \log r)} \leq \frac{(1+\delta)c_k}{2r} 
\leq \frac{(1+\delta)^2}{2}\frac{1}{\abs{z}} 
\leq \frac{\frac12 + \epsilon}{\abs{z}}.
\end{equation}
To show \eqref{eq:decayuk}, we see that that for all $\rho_{k,2}\leq \abs{z} \leq  \rho_{k,1} < 1$
\begin{equation}
\abs{u_k(z)} \leq \abs{z}^{\frac12 \varphi_k(|z|)} \abs*{\mathbf{E}_k(z)} \leq \abs{ \mathbf{E}_k(z) } = |z|^k.
\end{equation}

\subsubsection{Definition of the functions $u_k$ and $\mathbb V_k$ in $\{\rho_{k,3} \leq
  \abs{z}\leq\rho_{k,2}\}$}
  \label{subsec:32.2D}
We let 
\begin{equation}
\phi_k(r) := \chi_\delta \left( \frac{\log \rho_{k,2}-\log r}{\log \rho_{k,2} - \log \rho_{k,3}} \right) \quad \text{ for all }r>0.
\end{equation}
We have that $\phi_k \in C^\infty(0,+\infty)$, $\phi_k(\rho_{k,3}) =1$, $\phi_k(\rho_{k,2}) =0$, $\phi_k$ is decreasing and $\supp \phi_k' \subset  [\rho_{k,3}, \rho_{k,2}]$.

For $z  = r e^{i\theta}$, $r \in  [\rho_{k,3},\rho_{k,2}]$ and $\theta \in [-\pi,\pi)$, we set
\begin{align}
\notag
 u_k (z)  & := |z|^{1/2 }     \mathbf{E}_{k}(z) + \phi_k(\abs{z})  \frac{\mathbf{E}_{k+1}(z)}{\abs{z}^{1/2}}, \\
\label{eq:defn.Vk.32}
\V_k(z) & := 
\begin{pmatrix}
0 & \frac{-ie^{-i\theta}}{2r} \\
\frac{ie^{i\theta}}{2r} & -2i \partial_{\bar{z}} (\phi_k(r)) \frac{z^{k+1}}{\abs{z}\bar{z}^k}
\end{pmatrix}.
\end{align}
To show \eqref{eq:main2dlocal}, we compute explicitely $\D u_k(z)$ thanks to \eqref{eq:partialz.in.radial}:
\begin{equation*}
\begin{split}
\D u_k(z) = &
\begin{pmatrix}
0 & -2i \partial_z \\
-2i \partial_{\bar{z}} & 0
\end{pmatrix}
\begin{pmatrix}
\phi_k(|z|) \frac{ z^{k+1}}{|z|^{1/2}} \\
|z|^{1/2} \bar{z}^k 
\end{pmatrix}
\\ = &
\begin{pmatrix}
-2i \partial_z(|z|^{1/2}) \bar{z}^k  \\
-2i \partial_{\bar{z}}(\phi_k(|z|))  \frac{ z^{k+1}}{|z|^{1/2}} -2i \phi_k(r) \partial_{\bar{z}}( |z|^{-1/2} )z^{k+1} 
\end{pmatrix}
\\ = &
\begin{pmatrix}
-i \frac{e^{-i\theta}}{2r} |z|^{1/2} \bar{z}^k  \\
-2i \partial_{\bar{z}}(\phi_k(|z|))  \frac{ z^{k+1}}{|z|^{1/2}} + i \frac{e^{i\theta}}{2r} \phi_k(r) \frac{  z^{k+1} }{|z|^{1/2}}
\end{pmatrix}
\end{split}
\end{equation*}
so \eqref{eq:main2dlocal} holds true for all $r = |z| \in [\rho_{k,3},\rho_{k,2}]$.

Thanks to \Cref{lem:cut.off}\ref{itm:2} and  \eqref{eq:partialz.in.radial}, we have that
\begin{equation}\label{eq:choose.k0.32}
\begin{split}
\abs*{-2i \partial_{\bar{z}} (\phi_k(r)) \frac{z^{k+1}}{\abs{z}\bar{z}^k}}  
 & = \abs{\phi_k'(r)} 
 \leq \frac{\norm{\chi_\delta'}_{L^\infty(\R)}}{r} \frac{1}{\log\rho_{k,2} - \log\rho_{k,3}} 
 \\ & \leq \frac{1+\delta}{r} \frac{1}{e^{(k+\frac12)^2}(1-e^{-\frac{k}{3} -\frac{5}{36}})}
 \leq \frac{\delta(1+\delta)}{r},
 \end{split}
\end{equation}
choosing $k_0 \in \N$ big enough that
\begin{equation}\label{eq:choice.k0.32}
\frac{1}{e^{(k+\frac12)^2}(1-e^{-\frac{k}{3} -\frac{5}{36}})} \leq \delta \quad \text{ for all }k \geq k_0.
\end{equation}
For $r= |z| \in [\rho_{k,3},\rho_{k,2}]$, thanks to \eqref{eq:defn.Vk.32} and \eqref{eq:choose.k0.32} we have
\begin{equation}
\begin{split}
|\V(z)| 
\leq &
\left\vert
\begin{pmatrix}
0 & \frac{-ie^{-i\theta}}{2r} \\
\frac{ie^{i\theta}}{2r} & 0
\end{pmatrix} 
\right\vert
+
\left\vert
\begin{pmatrix}
0 & 0 \\
0 & -2i \partial_{\bar{z}} (\phi_k(r)) \frac{z^{k+1}}{\abs{z}\bar{z}^k} 
\end{pmatrix}
\right\vert
\\
= &
\frac{1}{2r}  + \abs*{-2i \partial_{\bar{z}} (\phi_k(r)) \frac{z^{k+1}}{\abs{z}\bar{z}^k}}  
\leq  \frac{\frac12 + \delta(1+\delta)}{r}.
\end{split}
\end{equation}
We get \eqref{eq:decayVk} thanks to \eqref{eq:delta.small}.

To show \eqref{eq:decayuk}, we see that that for all $\rho_{k,3}\leq \abs{z} \leq  \rho_{k,2} < 1$
\begin{equation}
\abs{u_k(z)} \leq |z|^{1/2 }   \abs*{  \mathbf{E}_{k}(z) }+ \abs{\phi_k(\abs{z})} \,  \frac{\abs*{\mathbf{E}_{k+1}(z) }}{\abs{z}^{1/2}} \leq  2 |z|^{k+\frac12} \leq 2|z|^k.
\end{equation}

  \subsubsection{Definition of the functions $u_k$ and $\mathbb V_k$ in $\{\rho_{k,4} \leq
  \abs{z}\leq\rho_{k,3}\}$}
    \label{subsec:43.2D}

The construction is similar to the one in   \Cref{subsec:32.2D}.
We let 
\begin{equation}
\widetilde\phi_k(r) := \chi_\delta \left( \frac{\log r - \log \rho_{k,4} }{\log \rho_{k,3} - \log \rho_{k,4}} \right) \quad \text{ for all }r>0.
\end{equation}
We have that $\widetilde\phi_k \in C^\infty(0,+\infty)$, $\widetilde\phi_k(\rho_{k,4}) =0$, $\phi_k(\rho_{k,3}) =1$, $\phi_k$ is increasing and $\supp {\widetilde\phi_k}' \subset  [\rho_{k,4}, \rho_{k,3}]$.

For $z  = r e^{i\theta}$, $r \in  [\rho_{k,4},\rho_{k,3}]$ and $\theta \in [-\pi,\pi)$, we set
\begin{align}
\notag
 u_k (z)  & := \widetilde\phi_k(\abs{z}) |z|^{1/2}     \mathbf{E}_{k}(z) +  \frac{ \mathbf{E}_{k+1}(z)}{\abs{z}^{1/2}}, \\
\label{eq:defn.Vk.43}
\V_k(z) & := 
\begin{pmatrix}
-2i \partial_z (\widetilde\phi_k(r)) \frac{\abs{z}\bar{z}^k}{z^{k+1}} & \frac{-ie^{-i\theta}}{2r} \\
\frac{ie^{i\theta}}{2r} & 0
\end{pmatrix}.
\end{align}
With similar computations as in  \Cref{subsec:32.2D},  \eqref{eq:main2dlocal} follows.

We observe that, thanks to \Cref{lem:cut.off}\ref{itm:2} and  \eqref{eq:partialz.in.radial}, 
\begin{equation}\label{eq:choose.k0.43}
\begin{split}
\abs*{-2i \partial_z (\widetilde\phi_k(r)) \frac{\abs{z}\bar{z}^k}{z^{k+1}  }}
 & = \abs{\widetilde\phi_k'(r)} 
 \leq \frac{\norm{\chi_\delta'}_{L^\infty(\R)}}{r} \frac{1}{\log\rho_{k,3} - \log\rho_{k,4}} 
 \\ & \leq \frac{1+\delta}{r} \frac{1}{e^{(k+\frac23)^2}(1-e^{-\frac{k}{3} +\frac{5}{36}})}
 \leq \frac{\delta(1+\delta)}{r},
 \end{split}
\end{equation}
choosing $k_0 \in \N$ big enough that
\begin{equation}\label{eq:choice.k0.43}
\frac{1}{e^{(k+\frac23)^2}(1-e^{-\frac{k}{3} +\frac{5}{36}})} \leq \delta \quad \text{ for all }k \geq k_0.
\end{equation}

For $r= |z| \in [\rho_{k,4},\rho_{k,3}]$, thanks to \eqref{eq:defn.Vk.43} and \eqref{eq:choose.k0.43} we have
\begin{equation}
\begin{split}
|\V(z)| 
\leq &
\left\vert
\begin{pmatrix}
0 & \frac{-ie^{-i\theta}}{2r} \\
\frac{ie^{i\theta}}{2r} & 0
\end{pmatrix} 
\right\vert
+
\left\vert
\begin{pmatrix}
-2i \partial_z (\widetilde\phi_k(r)) \frac{\abs{z}\bar{z}^k}{z^{k+1}}  & 0 \\
0 & 0
\end{pmatrix}
\right\vert
\\
= &
\frac{1}{2r}  + \abs*{-2i \partial_z (\widetilde\phi_k(r)) \frac{\abs{z}\bar{z}^k}{z^{k+1}} }  
\leq  \frac{\frac12 + \delta(1+\delta)}{r}.
\end{split}
\end{equation}
We get \eqref{eq:decayVk} thanks to \eqref{eq:delta.small}.
To show \eqref{eq:decayuk}, we see that that for all $\rho_{k,4}\leq \abs{z} \leq  \rho_{k,3} < 1$
\begin{equation}
\abs{u_k(z)} \leq    \abs{\widetilde\phi_k(\abs{z})} \,    |z|^{1/2 }   \abs*{  \mathbf{E}_{k}(z) }+ \frac{\abs*{\mathbf{E}_{k+1} }}{\abs{z}^{1/2}} 
\leq  2 |z|^{k+\frac12} \leq  2 |z|^k.
\end{equation}

\subsubsection{Definition of the functions $u_k$ and $\mathbb V_k$ in $\{\rho_{k,5} \leq
  \abs{z}\leq\rho_{k,4}\}$}
    \label{subsec:54.2D}
    The construction is similar to the one in   \Cref{subsec:21.2D}.
We let 
\begin{align}
& \widetilde\varphi_k(r):=\chi_\delta \left( \tilde{c}_k \left(1 - \frac{\log \rho_{k,4}}{\log r}\right)\right)-1 \quad \text{ for all } r \in \R,
\\
\label{eq:defn.tildeck}
& \tilde{c}_k:= \left(1 - \frac{\log \rho_{k,4}}{\log \rho_{k,5}}\right)^{-1}>0.
\end{align}
We have that $\widetilde\varphi_k \in C^\infty(0,+\infty)$, $\widetilde\varphi(\rho_{k,5})=0$, $\widetilde\varphi(\rho_{k,4})=-1$, $\widetilde\varphi_k$ is decreasing and $\supp \widetilde\varphi_k' \subset [\rho_{k,5}, \rho_{k,4}]$.

For $z  = r e^{i\theta}$, $r \in  [\rho_{k,5},\rho_{k,4}]$ and $\theta \in [-\pi,\pi)$, we set 
\begin{align}
\notag
 u_k (z)  & := |z|^{\frac12 \widetilde\varphi_k(r)}     \mathbf{E}_{k+1}(z), \\
\label{eq:defn.Vk.54}
\V_k(z) & := 
\begin{pmatrix}
0 & 0 \\
\frac{-i e^{-i\theta}}{2} \frac{d}{dr}(\widetilde\varphi_k(r) \log r) & 0
\end{pmatrix}.
\end{align}
Analogously to \Cref{subsec:21.2D}, \eqref{eq:main2dlocal} follows by explicit computation.

We show now \eqref{eq:decayVk}.  For all $\rho_{k,5} \leq r \leq \rho_{k,4} < 1$
\begin{equation*}
\begin{split}
 & \abs*{\frac{-ie^{-i\theta}}{2} \frac{d}{dr} \left(\widetilde\varphi_k(r) \log r \right) }
\\
 & = \frac{1}{2r} \abs*{ \chi_\delta' \left( \tilde{c}_k \left(1 - \frac{\log \rho_{k,4}}{\log r}\right)\right)\frac{\tilde{c}_k \log \rho_{k,4}}{\log r} + \chi_\delta \left( \tilde{c}_k \left(1 - \frac{\log \rho_{k,4}}{\log r}\right)\right)   -1 }.
 \end{split}
\end{equation*}
Thanks to \Cref{lem:cut.off}\ref{itm:1},\ref{itm:2},
\begin{equation*}
\begin{split}
 &  \frac{1}{2r} \left[ \chi_\delta' \left( \tilde{c}_k \left(1 - \frac{\log \rho_{k,4}}{\log r}\right)\right)\frac{\tilde{c}_k \log \rho_{k,4}}{\log r} + \chi_\delta \left( \tilde{c}_k \left(1 - \frac{\log \rho_{k,4}}{\log r}\right)\right)   -1 \right]
 \\
 & \leq \frac{ (1+\delta)}{2r} \frac{\tilde{c}_k \log \rho_{k,4}}{\log r} \leq \frac{ (1+\delta)\tilde{c}_k}{2r}
\end{split}
\end{equation*}
and
\begin{equation*}
 \frac1{2r} \left[ 1- \chi_\delta' \left( \tilde{c}_k \left(1 - \frac{\log \rho_{k,4}}{\log r}\right)\right)\frac{\tilde{c}_k \log \rho_{k,4}}{\log r} - \chi_\delta \left( \tilde{c}_k \left(1 - \frac{\log \rho_{k,4}}{\log r}\right)\right)   \right]
  \leq \frac{1}{2r},
\end{equation*}
so we conclude that for all $|z| = r \in [\rho_{k,5},\rho_{k,4}]$ 
\begin{equation*}
\begin{split}
\abs{\V_k(z)} & = \abs*{\frac{-i e^{-i\theta}}{2} \frac{d}{dr}(\widetilde\varphi_k(r) \log r)}
\\
& \leq \frac{1+\delta}{2r}\tilde{c}_k = \frac{1+\delta}{2r} \left(1+\frac{1}{e^{\frac{k}{3} +\frac14}-1} \right) \leq \frac{(1+\delta)^2}{2r} \leq \frac{\frac12 +\epsilon}{r},
\end{split}
\end{equation*}
choosing $k_0 \in \N$ such that
\begin{equation}\label{eq:choice.k0.54}
\frac{1}{e^{\frac{k}{3} +\frac14}-1} \leq  \delta \quad \text{ for all }k\geq k_0,
\end{equation}
and thanks to \eqref{eq:delta.small}.
To show \eqref{eq:decayuk}, we see that that for all $\rho_{k,5}\leq \abs{z} \leq  \rho_{k,4} \leq 1$
\begin{equation}
\abs{u_k(z)} \leq \abs{z}^{\frac12 \widetilde\varphi_k(|z|)} \abs*{\mathbf{E}_{k+1}(z)} \leq \frac{\abs{ \mathbf{E}_{k+1}(z) }}{|z|^{1/2}} 
= |z|^{k+\frac12} \leq |z|^k.
\end{equation}

\subsection{Definition of $u_0$ and $\V_0$ in $\{
  \abs{z}\geq \rho_{k_0}\}$}
\label{sec:defn.infinity.2D}
In this section we assume $k_0$ to be even: the
case of odd $k_0$ can be treated analogously.
We let 
\begin{equation}
\eta(r) := \chi_\delta \left(\frac{\log r}{\log \rho_{k_0}}\right)\quad \text{ for all }r>0.
\end{equation}
We have that $\eta \in C^\infty(0,+\infty)$, $\eta(\rho_{k_0}) =1$, $\eta(1) =0$, $\eta$ is decreasing and $\supp \eta' \subset  [\rho_{k_0}, 1]$.

For $z  = r e^{i\theta}$, $r \in  [\rho_{k_0},+\infty)$ and $\theta \in [-\pi,\pi)$, we set
\begin{align}
\notag
u_0(z)  & := \eta(|z|) \mathbf{E}_{k_0}(z), \\
\label{eq:defn.V0.}
\V_0(z) & := 
\begin{pmatrix}
0 & -i e^{-i\theta} \eta'(r) \\
0 & 0
\end{pmatrix}.
\end{align}
It is immediate to see that $u_0 \in C^\infty( \{ \abs{z} \geq \rho_{k_0} ; \C^2\})$, $\V_0 \in C^\infty( \{ \abs{z} \geq \rho_{k_0} ; \C^{2\times 2}\})$, that they vanish  for $\abs{z}\geq 1$ and that \eqref{eq:main2dlocal} holds true for $|z|\geq \rho_{k_0}$.
Moreover, thanks to \Cref{lem:cut.off}\ref{itm:2} and \eqref{eq:delta.small}
\begin{equation}\label{eq:decay.V0.2D}
 |\V_0 (z)| = \abs{\eta'(r)} \leq (1+\delta) \frac{1}{\abs{\log \rho_{k_0}} } \frac{1}{r} \leq  \frac{1+\epsilon}{2r} \leq \frac{\frac12 + \epsilon}{|z|}
 \quad \text{ for all }\abs{z} = r \geq \rho_{k_0},
\end{equation}
choosing $k_0 \in \N$ such that
\begin{equation}\label{eq:choice.k0.0}
\frac{1}{\abs{\log \rho_{k_0}}} = \frac{1}{e^{k_0^2}}\leq \frac12.
\end{equation}

\subsection{Definition of $u$ and $\V$ and conclusion of the proof}
\label{sec:defn.u.V.2D}
To define $u$ and $\V$, we exploit \eqref{eq:decomposition.R2} and set $u(z):=u_0(z)$, $\V(z):=\V_0(z)$ for $\abs{z}\geq \rho_0$, 
and $u(z):=u_k(z)$, $\V(z):=\V_k(z)$ for the appropriate $k$ such that $|z| \in [\rho_{k+1} , \rho_k]$. Finally, we set $u(0):=0$.

From the construction in the previous sections, we see that $\V \in C^\infty(\R^2\setminus\{0\}; \C^{2 \times 2})$
and $\D u(z) = \V(z)u(z)$ for $z \in \R^2\setminus\{0\}$. Thanks to \eqref{eq:decayVk} and \eqref{eq:decay.V0.2D}, $\abs{\V(z)} \leq (\frac12+ \epsilon)\frac{1}{|z|}$ for all $|z|>0$, so we conclude that \eqref{eq:main} holds true in $\R^2 \setminus\{0\}$, for $C := \frac12+\epsilon$. 

In the following lemma we determine the behaviour of the function $u$ in the origin.
\begin{lemma}\label{lem:decay.u.origin}
For all $k \geq k_0$ and for all $|z| \leq \rho_k$, $\abs{u(z)} \leq 2 \abs{z}^k$.
\end{lemma}
\begin{proof}
Let $k \geq k_0$ and $\abs{z} \in [0, \rho_k]$. If $z=0$, the thesis is trivially true since $u(0)=0$. If $|z| \in (0, \rho_k]$, thanks to \eqref{eq:decomposition.R2} there exists $h \geq k$ such that $|z| \in [\rho_{h+1},\rho_h]$. From \eqref{eq:decayuk},  $\abs{u(z)}=\abs{u_h(z)} \leq 2|z|^h \leq 2|z|^k$. 
\end{proof}
From the construction in the previous sections, we see that $u \in C^\infty(\R^2\setminus\{0\}; \C^2)$ and $\supp u \subset \{|z|\leq 1\}$. Moreover, thanks to 
\Cref{lem:decay.u.origin}, $u \in C(\R^2;\C^2)$ and $|z|^{-1}u \in L^2(\{|x| < \rho_k\};\C^N)$ for $k \geq k_0\geq 1$. 
To conclude that $u \in H^1(\R^n;\C^N)$, it is enough to show that it is in $H^1$  in a neighborhood of the origin.
This follows from the following Lemma. 
\begin{lemma}\label{lem:Hloc1.origin}
Let $n\geq 2$. Let $u$ such that $u \in H_{loc}^1 (\R^n \setminus \{ 0 \};\C^N)$ and $|x|^{-\alpha} u \in L^2 (\{ |x| < \delta \};\C^N)$ for $\alpha,\delta>0$. If there exists $C>0$ such that for a.a. $x \in \R^n$
\begin{equation}\label{eq:main.regularity}
\abs{\D_n u (x)} \leq \frac{C}{ |x|^{\alpha}} \abs{u(x)},
\end{equation}
 then $u \in H_{loc}^1(\R^n;\C^N)$.
\end{lemma}
\begin{proof}
It is enough to show that $u \in H^1(\{ |x|<\delta \};\C^N)$. From \eqref{eq:main.regularity} and since $|x|^{-\alpha} u \in L^2 (\{ |x| < \delta \};\C^N)$, we have  that $\D_n u \in L^2(\{ |x| < \delta \};\C^N)$. 
Let  $\chi \in C_c^\infty(\R^n;[0,1])$ such that $\chi(x) = 1$ for $|x|<\delta$:  we have that $\D_n( \chi u) \in L^2(\R^n;\C^N)$. Integrating by parts, using the fact that $(\D_n)^2 = -\Delta$ and integrating again by parts, we conclude that $|\nabla (\chi u)| \in L^2(\R^n;\R)$, so $|\nabla u| \in L^2(\{ |x| < \delta \} ;\R)$. Since $u \in L_{loc}^2(\R^n;\C^N)$, this gives $u \in H^1(\{ |x|<\delta \};\C^N)$.
\end{proof}

We finally show \eqref{eq:vanishing.infinite.order.origin}. Let $k \in N$: thanks to \Cref{lem:decay.u.origin}, for $R < \rho_k$ we have
\begin{equation}
\frac{1}{R^k} \int_{\{|x|<R\}} |u(x)|^2 \, dx \leq \frac{4}{R^k} \int_{\{|x|<R\}}  \abs{x}^{2k} \, dx \leq 4 R^k |\{|x| \leq R\}| \xrightarrow[R\to 0]{} 0.
\end{equation}
The proof of \Cref{thm:main} is complete.

\section{Proof of \Cref{theo:gamma}, \Cref{theo:gamma.example} and \Cref{thm:SUCinf}}
\label{sec:kelvin}
In this section we exploit a version of the Kelvin transformation adapted to the Dirac operator from
\cite[Section 4]{borrelli2020sharp}, \cite[Lemma 6.1]{jeong2022carleman}. 

For $u:\R^n \to \C^N$, the \emph{Kelvin transform} of $u$ is 
\begin{equation}\label{eq:defn.Kelvin}
\begin{split}
& u_\Kel : \R^n\setminus\{0\} \to \C^N
\\
&u_\Kel (x) := \frac{1}{\abs{x}^{n-1}} \left[ i  \alpha \cdot \frac{x}{|x|} \, \alpha_{n+1}\right] \, u \left(\frac{x}{\abs{x}^2}\right)
\quad \text{ for all }x\in \R^n\setminus\{0\},
\end{split}
\end{equation}
where we have used the shorthand notation $\alpha\cdot A = \sum_{j=1}^n A_j \alpha_j$ for all $A=(A_1,\dots,A_n) \in \C^n$.
It is immediate to see that $\norm{u}_{L^2(\R^n)} = \norm{u_\Kel}_{L^2(\R^n)}$. Moreover, we need the following lemma from \cite{borrelli2020sharp}.
\begin{lemma}[\cite{borrelli2020sharp}, Lemma 4.1]\label{lem:Kelvin}
Let $u:\R^n \to \C^N$, then
\begin{equation}
\D_n u_\Kel (x) = \frac{1}{|x|^2}  [\D_n u]_\Kel (x) \quad \text{ for all }x \in \R^n\setminus\{0\}.
\end{equation}
\end{lemma}
\subsection{Proof of \Cref{theo:gamma} }
Let  $u \in H_{loc}^1(\Omega_\rho;\C^n)$ such that \eqref{eq:gamma} holds true. Let $\psi:=u_\Kel \in H_{loc}^1(\{|x|>\rho\};\C^N)$.
Thanks to \Cref{lem:Kelvin}, for almost all $|x|>\rho$
\begin{equation}
\D_n \psi (x) 
=  \frac{1}{|x|^2}  [\D_n u]_\Kel (x) = \frac{1}{|x|^2} \frac{1}{|x|^{n-1}} \left[ i  \alpha \cdot \frac{x}{\abs{x}} \, \alpha_{n+1}\right] (\D_n u)\left(\frac{x}{|x|^2}\right).
\end{equation}
From the previous equation and \eqref{eq:gamma}, since $i  \alpha \cdot \tfrac{x}{\abs{x}} \, \alpha_{n+1}$ is unitary and $(i  \alpha \cdot \tfrac{x}{\abs{x}} \, \alpha_{n+1})^2 = \mathbb{I}_N$,
\begin{equation}\label{eq:defn.Kelvin.transformation}
\abs*{\D_n \psi (x)} \leq \frac{1}{|x|^2} \frac{C}{\abs*{\frac{x}{|x|^2}}^\gamma} \abs*{\frac{1}{|x|^{n-1}}u\left(\frac{x}{|x|^2}\right)}
= \frac{C}{|x|^{2-\gamma}} \abs{\psi(x)} \quad \text{ for a.a. }|x|>\rho.
\end{equation}
Since $2-\gamma<1$, we are in the assumptions of \cite[Theorem 1.1]{cassano2022sharp}, so there exists $\tau_0>0$ such that if $e^{\tau |x|^{2-2(2-\gamma)}}\psi = e^{\tau |x|^{2\gamma-2}}\psi\in L^2{(\{|x|>\rho\};\C^N)}$ for all $\tau>\tau_0$ then $\psi$ has compact support, but this is equivalent to say that if \eqref{eq:theo.power.gamma} holds true, then $u$ vanishes in a neighborhood of $0$. The theorem is then proven.

\subsection{Proof of \Cref{theo:gamma.example}}
Let $\gamma>1$ and $n\in\{2,3\}$. Thanks to \cite[Theorem 1.5]{cassano2022sharp}, for $\epsilon=2-\gamma<1$ there exists $\psi \in C^\infty(\R^n;\C^N)$, not compactly supported, such that for all $x \in\R^n$
\begin{align}
\notag
\abs{\psi(x)} & \leq C_1 e^{-C_2 |x|^{2-2\epsilon}} \\
\label{eq:solution.example}
\abs*{\D_n \psi(x)} & \leq 
\begin{cases}
C_3 |x|^{-\epsilon}\,\abs{\psi(x)} \quad &\text{ for }n=2, \\
C_3 |x|^{-\epsilon} (\log|x|)^3 \,\abs{\psi(x)} \quad &\text{ for }n=3,
\end{cases}
\end{align}
for some $C_1,C_2,C_3>0$.
We let $u:=\psi_\Kel$ and extend it by continuity to $0$ in the origin, getting $u \in C(\R^n;\C^N) \cap C^\infty(\R^n\setminus\{0\};\C^N)$, not vanishing in a neighborhood of the origin.
From \cite[eq.~(3.19)]{cassano2022sharp}, 
$\psi(z) = (C z^{m}, 0)^t$ for $|z| \in [0,\delta]$ for $C>0$, $\delta \in(0,1)$ and $m \in \N$ large. We can assume that $m$ is sufficiently big, in such a way that
$ u(z) = u(z_1,z_2)= \psi_\Kel(z) = \left[ i  \alpha \cdot \frac{(z_1,z_2)^t}{\abs{(z_1,z_2)^t}} \, \alpha_{n+1}\right] \frac{(C z^{m}, 0)^t}{|z|^{2m + n-1}}  $ is in $H^1(\{ |x| > R\};\C^N)$, for $R := \delta^{-1}>1$. 
Thanks to \Cref{lem:Kelvin}, the conditions \eqref{eq:gamma.sharp.u} and \eqref{eq:gamma.sharp.V} are verified for $C_3>0$ and (possibly different) $C_1,C_2>0$. Thanks to \eqref{eq:gamma.sharp.u}, \eqref{eq:gamma.sharp.V}, and \Cref{lem:Hloc1.origin},  $u \in H^1(\R^2;\C^2)$. 
The theorem is proven.

\subsection{Proof of \Cref{thm:SUCinf}}
 Let $ \psi \in H^1_{loc} (\widetilde\Omega_\rho ; \C^N)$ be a solution to \eqref{eq:main.inf} for $C \in (0,\frac12)$. Let 
Let $u:=\psi_\Kel \in H_{loc}^1(\{0<|x|<\rho\};\C^N)$; thanks to \Cref{lem:Kelvin}, for almost all $0< |x| < \rho$
\begin{equation}
\abs{\D_n u (x) } \leq \frac{C}{|x|} \abs{u(x)}.
\end{equation}
Since \eqref{eq:vanishing.infinite.order.inf} holds true for $\psi$, then \eqref{eq:vanishing.infinite.order.origin} holds true for $u$, so 
$|x|^{-1} u$ is in $L^2$ in a neighborhood of the origin (see \cite[Lemma 5.2]{pan2024unique}). Thanks to \Cref{lem:Hloc1.origin}, $u 
\in H_{loc}^1(\{|x|<\rho\};\C^N)$. 
We are in the assumptions of \Cref{thm:DCOKY}, so $u$ vanishes in a neighborhood of the origin, consequently $\psi$ has compact support. 

We let $C>\frac12$. Thanks to \Cref{thm:main}, 
there exists a (non-trivial)  function $u \in H^1(\R^2;\C^2) \cap C^\infty(\R^2 \setminus \{0\};\C^2)$, such that $\supp u \subset \{ |x| \leq 1 \}$. Moreover, $u$ vanishes at infinite order at the origin as in \eqref{eq:vanishing.infinite.order.origin} and \eqref{eq:main} holds true in $\R^2\setminus\{0\}$. We let $\psi := u_\Kel \in C^\infty(\R^2;\C^2)$. It is immediate to see that  \eqref{eq:main.inf} holds true in $\R^2\setminus\{0\}$ and \eqref{eq:vanishing.infinite.order.inf} holds true for $\psi$. 
The proof of the theorem is complete.

\printbibliography

\end{document}